\documentclass[a4paper,draft]{amsart}
\usepackage[a4paper,margin=25mm]{geometry}
\usepackage{amsmath}
\usepackage{amssymb,tikz,url}
\usepackage[backend=biber]{biblatex}

\newtheorem{lemma}{Lemma}
\newtheorem{theorem}{Theorem}

\DeclareMathOperator{\rht}{ht}
\DeclareMathOperator{\LIS}{LIS}

\begin{document}

\author[Manuel Kauers]{Manuel Kauers}
\address{Manuel Kauers, Institute for Algebra, J. Kepler University Linz, Austria}
\email{manuel.kauers@jku.at}

\author[Chen Wang]{Chen Wang}
\address{Chen Wang, Institute for Algebra, J. Kepler University Linz, Austria}
\email{wpolly0419@gmail.com}

\thanks{M. Kauers was supported by the Austrian FWF grants 10.55776/PAT8258123, 10.55776/PAT9952223,
  and 10.55776/I6130. C. Wang was supported by the Austrian FWF grant 10.55776/PAT9952223.}

\title{Recurrences for permutations with long increasing subsequences}

\begin{abstract}
  We prove two simple bivariate recurrences for the number of permutations with a long increasing subsequence.
  The two recurrences imply D-finiteness of the sequence in a certain range.
  As a consequence, we also obtain a proof of a conjecture posed by Kauers and Koutschan in 2023.
\end{abstract}

\maketitle

\section{Introduction}

Let
\[
a_{n,k}:=\#\{\sigma\in S_n\mid \LIS(\sigma)=k\},
\]
where $\LIS(\sigma)$ denotes the length of the longest increasing subsequence of $\sigma$. We use the convention that $a_{n,k}=0$ when $k<1$ or $k>n$. The resulting triangular array appears in the OEIS as A047874 and starts like
\begin{center}
  \begin{tabular}{ccccccc}
    1 \\
    1 & 1 \\
    1 & 4 & 1 \\
    1 & 13 & 9 & 1 \\
    1 & 41 & 61 & 16 & 1 \\
    1 & 131 & 381 & 181 & 25 & 1 \\
    1 & 428 & 2332 & 1821 & 421 & 31 & 1
  \end{tabular}.
\end{center}
Gessel~\cite{MR1041448} showed that every column of this array is D-finite, i.e., for every fixed~$k$, the
sequence $(a_{n,k})_{n=0}^\infty$ satisfies a linear recurrence with polynomial coefficients in~$n$.
For every fixed~$n$, the sequence $(a_{n,k})_{k=0}^\infty$ is finite and therefore trivially D-finite.
However, both results combined are not enough to conclude that $a_{n,k}$ is D-finite as a bivariate
sequence. Recall~\cite{Kauers23} that this requires a recurrence of the form
\[
  p_0(n,k)a_{n,k} + p_1(n,k)a_{n+1,k} + \cdots + p_r(n,k)a_{n+r,k} = 0
\]
with certain polynomials $p_0,\dots,p_r$, as well as a recurrence of the form
\[
  q_0(n,k)a_{n,k} + q_1(n,k)a_{n,k+1} + \cdots + q_s(n,k)a_{n,k+s} = 0
\]
with certain polynomials $q_0,\dots,q_s$. Rather than allowing each column (or row) to have its own
recurrence, there must be a uniform recurrence that works for all columns (or rows) simultaneously.

There is no reason to expect that such recurrences exist for the
sequence~$a_{n,k}$, and we believe that there are no such recurrences that hold
for all $n$ and~$k$. The main result of this note is that the bivariate sequence
does become D-finite if we restrict its domain to the sector $n/2 \leq k \leq n$.
Here and throughout the paper, we use the term ``D-finite inside a region~$R$''
in the sense of ``There exists a D-finite array that coincides with the original
array inside~$R$''.
\begin{theorem}\label{thm}
The two recurrences
\begin{equation}\label{eq:rec}
a_{n+1,k}+n^2a_{n-1,k} = a_{n,k+1}+(2n-1)a_{n,k}+a_{n,k-1}
\end{equation}
and
\begin{equation}\label{eq:rec2}
(n+1)^2(a_{n,k}+a_{n,k+1}) = (n-k+1)a_{n+1,k}+(n+k+1)a_{n+1,k+1}
\end{equation}
are valid in the region $n/2\leq k$.
\end{theorem}

The recurrences announced in Theorem~\ref{thm} and proved below both involve shifts with respect
to both $n$ and~$k$. Recurrences involving only shifts in one of the directions, as required by the
definition of D-finiteness, can be easily obtained from them by computer algebra. For example,
if \texttt{basis} denotes a list containing the two operators corresponding to the two recurrences
stated in the theorem, then typing
\begin{verbatim}
    FindRelation[OreGroebnerBasis[basis], Support -> Table[S[n]^i, {i, 0, 4}]]
\end{verbatim}
and
\begin{verbatim}
    FindRelation[OreGroebnerBasis[basis], Support -> Table[S[k]^i, {i, 0, 4}]]
\end{verbatim}
in Koutschan's Mathematica package \texttt{HolonomicFunctions}~\cite{KoutschanHolonomicFunctions}
produces the desired recurrences effortlessly. We do not reproduce them here.

Similarly, we obtain a proof of Conjecture~17 in~\cite{2303.02793} as a corollary of Theorem~\ref{thm}.
Setting
\[
b_{n,k}:=\#\{\sigma\in S_n\mid \LIS(\sigma)\geq k\}=\sum_{i=k}^{\infty}a_{n,i},
\]
Conjecture~17 concerns the sequence $b_{2n,n}$, whose $n$th term is the number of permutations of length $2n$ containing an increasing subsequence of length~$n$.
This sequence appears as A269021 in the OEIS, and based on the terms given there, Kauers and Koutschan
conjectured that $b_{2n,n}$ is D-finite, satisfying a recurrence of order~4.

Using Koutschan's package, a recurrence for $b_{2n,n}$ can be derived from Theorem~\ref{thm} by typing
the following commands:
\begin{verbatim}
    G := OreGroebnerBasis[basis];
    Q := CreativeTelescoping[G, S[k]-1, {S[n]}][[2,1]];
    B := DFiniteOreAction[G, Q];
    FindRelation[B, Support -> Table[S[n]^(2i)S[k]^i, {i,0,4}]]
\end{verbatim}
Up to normalization, this produces precisely the recurrence stated in Conjecture~17 of~\cite{2303.02793}.

Our proof of Theorem~\ref{thm} is based on the Robinson--Schensted
correspondence. It expresses $a_{n,k}$ as a sum of squares of dimensions of
irreducible representations indexed by partitions with first row~$k$. In the
range $k\geq n/2$, these partitions have a sufficiently long first row that the
ordinary branching operators and the Murnaghan--Nakayama rim-hook operators take
a particularly simple form. Combining these relations with character
orthogonality will give \eqref{eq:rec} and~\eqref{eq:rec2}.

\section{Large scale calculation of $a_{n,k}$}

The recurrences in \eqref{eq:rec} and \eqref{eq:rec2} were found using guessing~\cite{KauersGuess} from the
first few sequence terms. 
In order to efficiently compute enough terms of the array $a_{n,k}$ to find these recurrences, we use Gessel's
determinantal formula \cite[Section 7]{MR1041448} for the cumulative counts
\[
u_k(n):=a_{n,1}+\dots+a_{n,k}.
\]

It states that (after suitable transformation of the original formula)
\begin{equation}\label{eq:gessel}
\sum_{n\geq0}u_k(n)\frac{x^{n}}{(n!)^2}=\det\left(x^{(i-j)/2}I_{i-j}(2\sqrt{x})\right)_{i,j=1,\ldots,k},
\end{equation}
where $I_\nu(z)$ is the modified Bessel function of the first kind, and the matrix entries have nice series expansions
\[
x^{s/2}I_{s}(2\sqrt{x}) = \begin{cases}
    x^s\sum_{m\geq0} \frac{x^{m}}{m!(m+s)!} & s\geq 0,\\
    \sum_{m\geq0} \frac{x^{m}}{m!(m-s)!} & s<0.
\end{cases}
\]

To calculate $u_k(n)$ for $n,k$ up to $N$ and $K$ utilizing \eqref{eq:gessel} requires computing the leading principal minors of a $K\times K$ matrix whose entries are truncated power series in $x$ of degree $N$. This is done by ordinary Gaussian elimination without pivoting. In order to prevent large integer arithmetic, we do the calculations modulo many primes (with each prime larger than $N+K$ and their product larger than $N!$) and then use the Chinese remainder theorem to recover $u_k(n)$. With this approach and naive truncated power series arithmetic, we can compute $u_k(n)$ for $n,k$ up to $N$ and $K$ in time $O(N^2K^3)$ per prime (and therefore $O(N^3K^3\log N)$ in total using fixed-word-size primes). Finally, we recover $a_{n,k}$ from $u_k(n)$ by taking differences.

The actual computation for $N=300, K=200$ took about 1 hour on an i5-12400 desktop computer, and produced enough data to guess the recurrences \eqref{eq:rec} and \eqref{eq:rec2}. For further investigation of the array $a_{n,k}$ (including guessing more complicated recurrences), we have also computed $u_k(n) \mod 2^{30}+3$ for $n,k$ up to $N=K=1200$, which took about 1 day on the same computer.

It turns out that on the other side of the diagonal $k=n/2$, the array $a_{n,k}$ also satisfies bivariate recurrences, but they are more complicated. Experimentally we have found recurrences with higher order and degree for the region $n/3<k<n/2$, the shortest among them is
\begin{align*}
    &(n+1)^2(n+2)^2\left(\left(-2 k^2+k+3\right) a_{n,k}-\left(2 k^2+k-3\right) a_{n,k+1}\right.\\
    &\qquad\left.+(2 k+1) a_{n,k-1}+(1-2 k) a_{n,k+2}\right)\\
    &+(n+2)^2\left((-6+k+5k^2-2k^3-6n-kn+6k^2n)a_{n+1,k}\right.\\
    &\qquad+(-6-k+5k^2+2k^3-6n+kn+6k^2n)a_{n+1,k+1}\\
    &\qquad\left. -(1+2k)(4-3k+2n)a_{n+1,k-1}-(1-2k)(4+3k+2n)a_{n+1,k+2}\right)\\
    &+\left((4n-2)k^3-(n^2+8n+7)k\right)\left(a_{n+2,k}-a_{n+2,k+1}\right)\\
    &+\left(-(6n^2+22n+19)k^2+(3n^2+9n+6)\right)\left(a_{n+2,k}+a_{n+2,k+1}\right)\\
    &+(2 k+1) (k-n-3) (2 k-n-2) a_{n+2,k-1}\\
    &-(2 k-1) (k+n+3) (2 k+n+2) a_{n+2,k+2}\\
    &+k (2 k-1) (k+n+3) a_{n+3,k+1}-k (2 k+1) (k-n-3) a_{n+3,k}=0,
\end{align*}
\begin{figure}
  \begin{center}
  \begin{tikzpicture}[scale=.25]
    \fill[gray](0,0)--(12,-12)--(12/2,-12)--cycle;
    \fill[lightgray](0,0)--(12/2,-12)--(12/3,-12)--cycle;
    \fill[white](0,0)rectangle(1,-3);
    \draw[<->](0,-12)node[below]{$n$}--(0,0)--(12,0)node[right]{$k$};
    \foreach\row in{1,...,12}
    \foreach\col in{1,...,\row}
    \draw(\col,-\row)node{$\cdot$};
    \footnotesize
    \draw(1,-12)node[below]{$1$}(12/3,-12)node[below]{$\frac n3$}(12/2,-12)node[below]{$\frac n2$}(12,-12)node[below]{$n$};
  \end{tikzpicture}\hfil
  \parbox[b]{.3\hsize}{\raggedright\small
    \textbf{Figure 1.}
    The triangular array $a_{n,k}$ is proven to be D-finite in the region
    marked in dark gray.
    In addition, we have a system of guessed recurrences which appear to
    hold for the sequence in the light gray region.

    \bigskip
    }
  \end{center}  
\end{figure}

Interestingly enough, these complicated recurrences also appear to hold for the
region $k\geq n/2$, and we have checked by direct reduction that they lie in the
(left) ideal generated by \eqref{eq:rec} and \eqref{eq:rec2}. We suspect that
for every $r=2,3,\dots$, the sequence $a_{n,k}$ is D-finite in the range $k\geq n/r$.
Note that even if this is true, it would not imply that $a_{n,k}$ is D-finite on
the whole range $k\geq0$.

Direct combinatorial proofs for the recurrences we found for $r=3$ are likely to be
very hard. One hopes that a uniform approach can be found to deal with all such
slices of the array $a_{n,k}$ and understand the resulting chain of ideals.
For the time being, we can only prove \eqref{eq:rec} and~\eqref{eq:rec2}.
The rest of the paper is dedicated to this proof. 

\section{Proof of Theorem~\ref{thm}}

Most of the standard notation and terminology of the representation theory of symmetric groups can be found in \cite{SaganSymmetricGroup} or \cite{StanleyEC2}.

By the Robinson--Schensted correspondence,
\begin{equation}\label{eq:aRSK}
 a_{n,k}
 =\sum_{\substack{\lambda\vdash n\\ \lambda_1=k}}(f^\lambda)^2
 =\sum_{\substack{\nu\vdash n-k\\\nu_1\leq k}}
    (f^{(k,\nu)})^2.
\end{equation}

We first define some auxiliary objects.  Let $x_\lambda$ be formal
basis vectors indexed by integer partitions, and define the real vector
space
\[
P_m:=\operatorname{span}\{x_\lambda\mid\lambda\vdash m\}
\]
for every $m\geq0$, equipped with the standard inner product.

For every positive integer $s$, let $D_s$ be the rim-hook removal
operator
\[
D_sx_\lambda:=
\sum_{\substack{\mu\vdash |\lambda|-s\\
\lambda\setminus\mu\text{ is a rim hook}}}
(-1)^{\rht(\lambda\setminus\mu)}x_\mu,
\]
and let its adjoint $U_s=D_s^*$ be the rim-hook addition operator
\[
U_sx_\mu:=
\sum_{\substack{\lambda\vdash |\mu|+s\\
\lambda\setminus\mu\text{ is a rim hook}}}
(-1)^{\rht(\lambda\setminus\mu)}x_\lambda.
\]


We also define the ``regular representation'' element
\[
R_N:=\sum_{\lambda\vdash N}f^\lambda x_\lambda\in P_N.
\]
The standard character restriction and induction relationship on symmetric groups can be represented succinctly by
\[
D_1R_{n+1}=(n+1)R_n,\qquad U_1R_n=R_{n+1}.
\]

For $l,m\geq0$, let
\[
W_{l,m}:=
\sum_{\substack{\lambda\vdash m\\\lambda_1\leq l}} f^{(l,\lambda)}x_\lambda\in P_m.
\]
Then \eqref{eq:aRSK} says that
\begin{equation}\label{eq:aW}
   a_{n,k}=\|W_{k,n-k}\|^2.
\end{equation}

\begin{lemma}\label{lem:parseval}
For every $V\in P_m$,
\begin{equation}\label{eq:parseval}
    \sum_{s=1}^m\|D_sV\|^2=m\|V\|^2.
\end{equation}
\end{lemma}

\begin{proof}
We write
\[
V=\sum_{\lambda\vdash m}v_\lambda x_{\lambda}.
\]
Standard character orthogonality implies
\begin{equation}\label{eq:vnorm}
\|V\|^2=\frac{1}{m!}\sum_{\sigma\in S_m}\left(\sum_{\lambda\vdash m}v_\lambda\chi_{\lambda}(\sigma)\right)^2.
\end{equation}

On the other hand, the Murnaghan--Nakayama rule gives
\begin{equation*}
    D_sV=\sum_{\lambda\vdash m}v_\lambda\sum_{\substack{\mu\vdash m-s\\
\lambda\setminus\mu\text{ is a rim hook}}}
(-1)^{\rht(\lambda\setminus\mu)}x_\mu,
\end{equation*}
and therefore
\begin{align}
    \|D_sV\|^2
    &=\frac{1}{(m-s)!}\sum_{\tau\in S_{m-s}}
      \left(\sum_{\substack{\lambda\vdash m\\ \mu\vdash m-s\\
      \lambda\setminus\mu\text{ is a rim hook}}}
      (-1)^{\rht(\lambda\setminus\mu)}v_\lambda\chi_\mu(\tau)\right)^2\nonumber\\
    &=\frac{1}{(m-s)!}\sum_{\tau\in S_{m-s}}
      \left(\sum_{\lambda\vdash m}v_\lambda\chi_\lambda(c_s\tau)\right)^2,
      \label{eq:dvnorm}
\end{align}
where $c_s$ is an $s$-cycle disjoint from $\tau$.

Combining \eqref{eq:vnorm} and \eqref{eq:dvnorm}, we see that
\eqref{eq:parseval} is equivalent to
\begin{equation*}
    \sum_{\sigma\in S_m}\left(\sum_{\lambda\vdash m}v_\lambda\chi_{\lambda}(\sigma)\right)^2
    =\sum_{s=1}^{m}\frac{(m-1)!}{(m-s)!}\sum_{\tau\in S_{m-s}}
      \left(\sum_{\lambda\vdash m}v_\lambda\chi_\lambda(c_s\tau)\right)^2.
\end{equation*}
This last identity can be proved bijectively.  Let $S_m$ act on
$\{1,2,\dots,m\}$.  There are exactly
\[
 \binom{m-1}{s-1}(s-1)! = \frac{(m-1)!}{(m-s)!}
\]
possible $s$-cycles containing $m$, and every permutation is uniquely
the product of its cycle containing $m$ and a permutation on the
remaining points.
\end{proof}

\begin{lemma}\label{lem:regular}
For $n\geq s\geq2$, one has $D_sR_n=0$.
\end{lemma}

\begin{proof}
Let $\varphi:S_s\times S_{n-s}\to S_n$ be the natural embedding, and
let $\pi$ be an $s$-cycle in $S_s$.  Since the character of the regular
representation is zero on every non-identity element, for every
$\sigma\in S_{n-s}$ we have
\begin{align*}
    0&=\sum_{\lambda\vdash n}f^\lambda
       \chi_{\lambda}(\varphi(\pi,\sigma))\\
    &=\sum_{\lambda\vdash n}f^\lambda
      \sum_{\substack{\mu\vdash n-s\\
      \lambda\setminus\mu\text{ is a rim hook}}}
      (-1)^{\rht(\lambda\setminus\mu)}\chi_{\mu}(\sigma)\\
    &=\sum_{\mu\vdash n-s}\chi_{\mu}(\sigma)
      \sum_{\substack{\lambda\vdash n\\
      \lambda\setminus\mu\text{ is a rim hook}}}
      (-1)^{\rht(\lambda\setminus\mu)}f^\lambda.
\end{align*}
Since the irreducible characters of $S_{n-s}$ are linearly independent
as class functions, it follows that
\[
\sum_{\substack{\lambda\vdash n\\
\lambda\setminus\mu\text{ is a rim hook}}}
(-1)^{\rht(\lambda\setminus\mu)}f^\lambda=0,
\qquad \mu\vdash n-s.
\]
This is exactly the $x_\mu$ coefficient of $D_sR_n$.
\end{proof}

In the long-first-row regime, Lemma~\ref{lem:regular} has a
particularly simple consequence.

\begin{lemma}\label{lem:Ds}
If $m\leq k+1$ and $2\leq s\leq m$, then
\begin{equation}\label{eq:Ds1}
    D_sW_{k,m}=-W_{k+s,m-s}.
\end{equation}
\end{lemma}

\begin{proof}
Fix $\lambda\vdash m-s$.  Since $s\geq2$, we have
$|\lambda|\leq k-1$, so $(k,\lambda)$ is a partition.  In the
coefficient of $x_{(k,\lambda)}$ in $D_sR_{k+m}$, an added rim hook
is either the horizontal hook producing $(k+s,\lambda)$ or lies
entirely below the first row and produces some $(k,\mu)$ occurring in
$W_{k,m}$.  A hook meeting both parts is impossible: once at least
one new cell is put in the first row, the lower part has at most $m-1\leq k$
cells and therefore cannot reach column $k+1$ to connect to it.
Consequently, the coefficient $[x_{(k,\lambda)}]D_sR_{k+m}$ (that vanishes because of Lemma~\ref{lem:regular}) is $f^{(k+s,\lambda)}$ plus the coefficient of $x_\lambda$ in $D_sW_{k,m}$, proving \eqref{eq:Ds1}.
\end{proof}

We shall also need the ordinary one-box branching relations in the
same notation.

\begin{lemma}\label{lem:D1U1}
For $l\geq1$,
\begin{align}
D_1W_{l,m}&=(l+m)W_{l,m-1}-W_{l+1,m-1},
&&1\leq m\leq l+1, \label{eq:D1W}\\
U_1W_{l,m}&=W_{l,m+1}-W_{l-1,m+1},
&&0\leq m<l. \label{eq:U1W}
\end{align}
\end{lemma}

\begin{proof}
Fix $\nu\vdash m-1$.  Under the assumption $m\leq l+1$, we have
$\nu_1\leq l$, and the coefficient of $x_\nu$ in $D_1W_{l,m}$ is the
sum of $f^{(l,\lambda)}$ over all ways of adding one box below the
first row of $(l,\nu)$.  The ordinary branching rule gives
\[
(l+m)f^{(l,\nu)}
=f^{(l+1,\nu)}+
\sum_{\substack{\lambda\vdash m,\ \lambda_1\leq l\\\lambda\setminus\nu
\text{ is one box}}}f^{(l,\lambda)},
\]
which proves \eqref{eq:D1W} coefficientwise.

For \eqref{eq:U1W}, fix $\mu\vdash m+1$.  Since $m<l$, one has
$\mu_1\leq l$.  The coefficient of $x_\mu$ in $U_1W_{l,m}$ is the
sum of $f^{(l,\lambda)}$ over the corners of $\mu$ whose removal gives
$\lambda$.  These are exactly the removals below the first row of the
shape $(l,\mu)$.  If $\mu_1<l$, the remaining first-row removal
contributes $f^{(l-1,\mu)}$; if $\mu_1=l$, that box is not a corner
and the coefficient of $x_\mu$ in $W_{l-1,m+1}$ is zero.  In either
case the branching rule gives precisely the coefficient of
$W_{l,m+1}-W_{l-1,m+1}$.
\end{proof}

\begin{proof}[Proof of \eqref{eq:rec}]
Put $m=n-k$.  First suppose that $1\leq m<k$.  From
\eqref{eq:U1W}, adjointness, and then \eqref{eq:D1W},
\begin{align*}
&\|W_{k,m+1}\|^2-\|W_{k-1,m+1}\|^2\\
&\quad=\langle U_1W_{k,m},
       W_{k,m+1}+W_{k-1,m+1}\rangle\\
&\quad=\langle W_{k,m},
       D_1W_{k,m+1}+D_1W_{k-1,m+1}\rangle\\
&\quad=\big\langle W_{k,m},
       nW_{k,m}+nW_{k-1,m}-W_{k+1,m}\big\rangle.
\end{align*}
Similarly,
\begin{align*}
&n^2\|W_{k,m-1}\|^2-\|W_{k+1,m-1}\|^2\\
&\quad=\langle D_1W_{k,m},
       nW_{k,m-1}+W_{k+1,m-1}\rangle\\
&\quad=\left\langle W_{k,m},
       U_1\bigl(nW_{k,m-1}+W_{k+1,m-1}\bigr)\right\rangle\\
&\quad=\big\langle W_{k,m},
       (n-1)W_{k,m}-nW_{k-1,m}+W_{k+1,m}\big\rangle.
\end{align*}
Adding the two identities cancels the mixed inner products and gives
\begin{align*}
&\|W_{k,m+1}\|^2+n^2\|W_{k,m-1}\|^2\\
&\qquad=\|W_{k+1,m-1}\|^2+(2n-1)\|W_{k,m}\|^2
       +\|W_{k-1,m+1}\|^2.
\end{align*}
By \eqref{eq:aW}, this is \eqref{eq:rec}.

The remaining case $m=0$, or $k=n$, is immediate from
$a_{n,n}=1$, $a_{n+1,n}=n^2$, and $a_{n,n-1}=(n-1)^2$.
\end{proof}

\begin{proof}[Proof of \eqref{eq:rec2}]
Again put $m=n-k$.  The case $m=0$ is immediate from
$a_{n,n}=a_{n+1,n+1}=1$ and $a_{n+1,n}=n^2$, so assume
$1\leq m\leq k$.

Apply Lemma~\ref{lem:parseval} and Lemma~\ref{lem:Ds} first to
$W_{k,m+1}$ and then to $W_{k+1,m}$.  We obtain
\begin{align*}
(m+1)\|W_{k,m+1}\|^2
 &=\|D_1W_{k,m+1}\|^2
   +\sum_{s=2}^{m+1}\|W_{k+s,m+1-s}\|^2,\\
m\|W_{k+1,m}\|^2
 &=\|D_1W_{k+1,m}\|^2
   +\sum_{s=3}^{m+1}\|W_{k+s,m+1-s}\|^2.
\end{align*}
After subtraction,
\begin{align}\label{eq:parseval-difference}
&(m+1)\|W_{k,m+1}\|^2-m\|W_{k+1,m}\|^2\nonumber\\
&\qquad=\|D_1W_{k,m+1}\|^2-\|D_1W_{k+1,m}\|^2
       +\|W_{k+2,m-1}\|^2.
\end{align}
Using \eqref{eq:D1W} twice, the right-hand side of
\eqref{eq:parseval-difference} is
\begin{align*}
&(n+1)^2\bigl(\|W_{k,m}\|^2-\|W_{k+1,m-1}\|^2\bigr)
 +\|W_{k+1,m}\|^2\\
&\quad+2(n+1)\bigl(
 \langle W_{k+1,m-1},W_{k+2,m-1}\rangle
 -\langle W_{k,m},W_{k+1,m}\rangle\bigr).
\end{align*}
The remaining pair of mixed products is eliminated by adjointness:
\begin{align*}
&(n+1)\|W_{k+1,m-1}\|^2
 -\langle W_{k+1,m-1},W_{k+2,m-1}\rangle\\
&\quad=\langle W_{k+1,m-1},D_1W_{k+1,m}\rangle\\
&\quad=\langle U_1W_{k+1,m-1},W_{k+1,m}\rangle\\
&\quad=\|W_{k+1,m}\|^2
 -\langle W_{k,m},W_{k+1,m}\rangle.
\end{align*}
Consequently,
\begin{align*}
&(m+1)\|W_{k,m+1}\|^2-m\|W_{k+1,m}\|^2\\
&\qquad=(n+1)^2\bigl(\|W_{k,m}\|^2
             +\|W_{k+1,m-1}\|^2\bigr)
       -(2n+1)\|W_{k+1,m}\|^2.
\end{align*}
Rearranging and using $2n+1-m=n+k+1$ gives
\begin{align*}
&(n+1)^2\bigl(\|W_{k,m}\|^2+\|W_{k+1,m-1}\|^2\bigr)\\
&\qquad=(m+1)\|W_{k,m+1}\|^2
 +(n+k+1)\|W_{k+1,m}\|^2.
\end{align*}
By \eqref{eq:aW}, this is exactly \eqref{eq:rec2}.
\end{proof}

\printbibliography

\end{document}